\documentclass[12pt,a4paper]{amsart}
\usepackage[margin=1in]{geometry}
\usepackage[cp1251]{inputenc}
\usepackage[T2A]{fontenc}     
\usepackage{amssymb}
\usepackage{mathrsfs}
\usepackage{latexsym}
\usepackage{amssymb,amsthm}
\usepackage{amsaddr}
\usepackage{float}
\newtheorem{theorem}{Theorem}
\newtheorem{lemma}{Lemma}

\newtheorem{assumption}{Assumption}
\theoremstyle{remark}

\newtheorem{example}{Example}

\begin{document}
\large
\thispagestyle{empty}
\title[Resource allocation in communication networks]{Resource allocation in communication networks with large number of users: the stochastic gradient descent method}
\author{D.B.\,Rokhlin}
\email{dbrohlin@sfedu.ru}
\address{Southern Federal University, Rostov-on-Don}
\date{}
\thanks{The research is supported by the Russian Science Foundation, project 17-19-01038.}
\maketitle
We consider a communication network with fixed number of links, shared by large number of users. The resource allocation is performed on the basis of an aggregate utility maximization in accordance with the popular approach, proposed by Kelly and coauthors \cite{Kelly1998}. The problem is to construct a pricing mechanism for transmission rates to stimulate an optimal allocation of the available resources.

In contrast to the usual approach, the proposed algorithm does not use the information on the aggregate traffic over each link. Its inputs are the total number $N$ of users, the link capacities and optimal myopic reactions of randomly selected users to the current prices. The dynamic pricing scheme is based on the dual projected stochastic gradient descent method. For a special class of utility functions $u_i$ we obtain upper bounds for the amount of constraint violation and the deviation of the objective function from the optimal value. These estimates are uniform in $N$ and are of order $O(T^{-1/4})$ in the number $T$ of reaction measurements. We present some computer experiments for quadratic utility functions $u_i$. 

{\it Kew words and phrases}: network utility maximization, duality, stochastic projected gradient descent method, large number of users\\ 

\section*{Introduction}
\setcounter{section}{0}
Contemporary communication networks contain large number of links, whose capacities are shared by huge number of users. Network resource management is aimed to optimally utilize the available resources,  prevent congestion and ensure the stability of the system. Furthermore, to be of practical value the control should be decentralized: users and links are considered as processors, updating their variables based on the dynamically monitored local information. Now conventional optimality criterion, the sum of user utilities, was proposed in \cite{Kelly1998}. In economic terms, this criterion can be called utilitarian, since it corresponds to the maximization of social welfare.

Consider a network with $m$ of links and $N$ users. Each user $i$ transmits packets over a fixed set of links. The network structure is determined by the routing matrix $R=(R_i^j)\in\mathbb R^{m\times N}$. Its columns $R_i\neq 0$, $i=1,\dots,N$ are binary $m$-dimensional vectors such that $R_i^j=1$, if the link $j$ is utilized by the user $i$ and $R_i^j=0$ otherwise. The link capacities are described by a vector $b\in\mathbb R^m$ with strictly positive components. The users evaluate the network quality by the utility functions $u_i(x^i)$, depending on the transmission rates $x^i\in\mathbb R_+$. An optimal resource allocation corresponds to an optimal solution $x^*\in\mathbb R^N_+$ of the network utility maximization (NUM) problem:  
\begin{align}
u(x) &=\sum_{i=1}^N u_i(x^i)\to\max,\label{1.1}\\
Rx &=\sum_{i=1}^N R_i x^i\le b,\quad x=(x^1,\dots,x^N)\in\mathbb R_+^N,\label{1.2}
\end{align} 
which was formulated in \cite{Kelly1998}. 

For given link prices $\overline\lambda\in\mathbb R^m_+$, the users select optimal transmission rates $\overline x_i$ maximizing the difference between the utility and price of $\overline x_i$:
\[ \overline x^i\in\arg\max_{x^i\in\mathbb R_+}\left(u_i(x^i)-x^i\sum_{j=1}^m\overline\lambda^j R_i^j \right).\]
The aim of the management is to stimulate this optimal resource allocation $x^*$ by setting the link prices $\lambda^*\in\mathbb R^m_+$. The research related to this problem, its variants and generalizations is reviewed in \cite{Srikant2004,Chiang2007,Shakkottai2008,Srikant2014}.

Under some technical assumptions, including the concavity of the utility functions $u_i$, the existence of the mentioned stimulating prices $\lambda^*$  follow from the duality theory. Moreover, these prices can be approximated using the dual projected gradient descent method: see \cite{Low1999,Nedic2010}. The related computations are completely distributed: each link $j$ updates its price $\lambda^j$ in accordance with the difference between the supply of $b^j $ and total demand
$\sum_{i=1}^N R^j_i x_i$, and each user $i$ updates the transmision rate $x^i$ on the basis of the link prices for the path $\{j:R_i^j=1\}$. 

In this paper the total traffic on the links is not assumed to be known. 
This problem statement makes sense, since the packets from users do not come simultaneously. The asynchronous model of \cite{Low1999} and the model with noisy feedback \cite{Zhang2008} differently address the same problem.

The input data for the dynamic pricing algorithm in question are (1) the total number of users $N$, (2) the link capacities $b$ and (3) the reactions $x^\xi$ of randomly selected users to current prices $\lambda$. It should be emphasized that this approach requires the knowledge of one global parameter: the number of users. The algorithm builds an approximation for the optimal price of $\lambda^*$ on the basis of relatively small number of user reactions. The sequential procedure for constructing this approximation is based on the dual projected \emph{stochastic} gradient descent method.

In section 2 for a special class of utility functions $u_i$ we give upper bounds for the amount of constraint violation and  the deviation of the objective function from the optimal value. These estimates are uniform in $N$ and are of orded $O(T^{-1/4})$ in the number $T$ of measured user reactions. Note that the fast gradient descent method of Nesterov \cite{Nesterov1983}, applied to the problem under consideration in \cite{Beck2014}, bounds the same quantities by $O(T^{-1})$ in the number $T$ of iterations. However, each iteration of the fast gradient descent method requires the knowledge of $N$ user reactions, if they are measured individually. So, for large values of $N$ the proposed algorithm may require much smaller number of user reactions measurements to achieve the desired accuracy. The computer experiments with quadratic utility functions $u_i$, presented in section \ref{sec:3}, illustrate this fact.

We assume that marginal utilities at zero $u_i '(0) $ are finite (and uniformly bounded by $N$). The consequence of this assumption is the fact that a significant proportion of users receive zero optimal data transmission rates. So, we interpret the proposed pricing mechanism as a way to manage an extra traffic. This means that initially each user receives a bandwidth of the order of $\min_{1\le j \le m} b^ j/N$, and only the remaining link capacities are shared according the proposed pricing scheme. However, in what follows we do not consider this aspect.

\emph{Notation.} We do not explicitly distinguish between row and column vectors. The scalar product and Euclidean norm are be denoted as follows: 
\[\langle x,y\rangle=\sum_{i=1}^k x^i y^i,\quad \|x\|=\sqrt{\langle x,x\rangle},\quad x,y\in\mathbb R^k. \]
We use lower indexes for vector numbers and upper indexes for their components. The gradient of a function is written as $g':=(g_{x^1},\dots,g_{x^k})$.

\section{The main result}
\label{sec:2}
\setcounter{equation}{0}
Let us briefly describe the standard approach to the NUM problem (\ref{1.1}), (\ref{1.2}). Consider the Lagrange function
\[ L(x,\lambda)=\sum_{i=1}^N u_i(x^i)+\langle\lambda,b-\sum_{i=1}^N R_i x^i\rangle,\quad (x,\lambda)\in\mathbb R_+^N\times\mathbb R_+^m\]
and the dual objective function
\begin{align*}
q(\lambda)=\sup_{x\in \mathbb R^N_+} L(x,\lambda)=\langle\lambda,b\rangle+\sum_{i=1}^N \sup_{x^i\in\mathbb R_+}(u_i(x^i)-\langle \lambda,R_i \rangle x^i),\quad \lambda\in\mathbb R_+^m,
\end{align*}
Denote by $x^*$ the solution of the primal problem (\ref{1.1}), (\ref{1.2}) and by $\lambda^*$ the solution of the dual problem 
\begin{equation} \label{2.0}
q(\lambda)\to\min_{\lambda\in\mathbb R_+^m}.
\end{equation}
Formally applying the Kuhn-Tucker conditions, we get the relations
\begin{align}
&x^{*,i}\in\arg\max_{x^i\in\mathbb R_+}(u_i(x^i)-\langle\lambda^*,R_i \rangle x^i),\label{2.1}\\
&\lambda^{*,j}\left(b^j-\sum_{i=1}^N R_i^j x^{*,i}\right)=0,\quad j=1,\dots,m.\label{2.2}\\
& \sum_{i=1}^N R_i x^{*,i}\le b,\quad x^{*,i}\in\mathbb R^n_+. \label{2.3} 
\end{align}

The Lagrange multipliers $(\lambda^{*,j})_{j=1}^m$ are interpreted as link prices, and $(R_i^j x^{*,i})_{j=1}^m$ are the optimal transmission rates of $i$-th user. The relations (\ref{2.1}) mean that $x^{*,i}$ are optimal reactions to the prices $\lambda^{*,j}$, $j\in L_i$ on the part of the selfish users, seeking for the ``revenue'' $u_i(x^i)-\langle\lambda^*,R_i x^i\rangle$ maximization. The conditions (\ref{2.3}) ensure the feasibility of $x^*$. The complementary slackness conditions (\ref{2.2}) imply that all links $j$ with non-zero prices $\lambda^{*,j}>0$ are completely utilized: $\sum_{i=1}^N R_i^j x^{*,i}=b^j$.

If the functions $-u_i$ are strongly convex, then the elementary problems
\begin{equation} \label{2.4}
\overline x^i(\lambda)\in \arg\max_{x^i\in\mathbb R_+}(u_i(x^i)-\langle\lambda,R_i \rangle x^i)
\end{equation}
have no more that one solution $\overline x^i(\lambda)$ for any $\lambda\in\mathbb R^m_+$. If such solutions exist, then the function $q$ is differentiable and
\begin{equation} \label{2.5}
q'(\lambda)=b-\sum_{i=1}^N R_i \overline x^i(\lambda)=b-R\overline x(\lambda).
\end{equation}
An optimal solution $\lambda^*$ can be computed by the dual projected gradient descent method:
\begin{equation} \label{2.6}
 \lambda_{t+1}=\left(\lambda_{t}-\eta_t (b-R\overline x(\lambda_t)) \right)^+,\quad t\ge 0,
 \end{equation}
where $\mu^+=(\max\{\mu^j,0\})_{j=1}^m$ and $\eta_t>0$ is a step sequence. If the aggregate demand $\sum_{i=1}^N R_i^j x^i$ is known, then each link $j$ can adjust its price $\lambda_t^j$ according to (\ref{2.6}), using only the local information on its resource demand. Under suitable technical conditions the sequence $\lambda_t$ converges to $\lambda^*$ and $\overline x(\lambda_t)$ converges to $x^*$. For the NUM problem this method was formulated in \cite{Low1999}, see also \cite{Beck2014,Nedic2010,Nesterov2018}.
  
As was already mentioned in the introduction, in this paper the quantities $b-R\overline x(\lambda_t)$ assumed to be unknown. To model the demands coming from random users, we replace the sum $\sum_{i=1}^N R_i \overline x^i(\lambda)$ by a random vector $N R_\xi \overline x^\xi(\lambda)$, where $\xi$ is uniformly distributed on $\{1,\dots,N\}$, and apply the projected stochastic gradient descent method. To bound the approximation errors uniformly in $N$, we impose the following conditions on the utility functions. 

\begin{assumption}  \label{as:1}
The functions $u_i$ are twice continuously differentiable on $(-\varepsilon,\infty)$, $\varepsilon>0$ and satisfy the conditions  
\[ u_i(0)=0,\quad 0<u_i'(0)\le B<\infty,\quad  1\le i\le N. \]
\end{assumption}

\begin{assumption}  \label{as:2}
The functions $-u_i$ are $(N\sigma)$-strongly convex:
\begin{equation} \label{2.7}
-u_i''(x^i) \ge N \sigma,\quad x^i\in\mathbb R_+,\quad \sigma>0.
\end{equation}
\end{assumption}

Although, as is clear from (\ref{2.7}), the functions $u_i$ depend on $N$, for the readability reasons we suppress this dependence in the notation. In what follows the Assumptions \ref{as:1}, \ref{as:2} are supposed to be fulfilled without further commentary.

Our main example is the quadratic utilities
\begin{equation} \label{2.8}
u_i(x^i)=a_i x^i-\frac{N \sigma_i}{2}(x^i)^2.
\end{equation}
To meet the Assumptions \ref{as:1}, \ref{as:2} we require that
\[ 0<a_i\le B,\quad \sigma_i\ge\sigma. \]
If the data trasmission is free, then such utility functions induce the individual demands of order $1/N$:
\[ \hat x^i=\frac{1}{N}\frac{a_i}{\sigma_i}.\] 
So, the users are ``small''. Since usually the resources are scarse, the aggregate demand $\sum_{i=1}^N R_i \hat x^i$ should exceed $b$ componentwise. 

On can regard (\ref{2.8}) as an approximation of a general $(N\sigma_i)$-strongly convex function near the origin. Furthermore, $a_i$ can be considered as a value of the unit transmission rate for the user $i$. The second term in (\ref{2.8}) can be regarded as a penalty, assigned by the network. From (\ref{2.4}) we get
\[ \overline x^i=\frac{(a_i-\langle\lambda,R_i\rangle)^+}{N\sigma_i}.\] 
Thus, the optimal rates (if positive) are proportional to the difference between the value coefficient $a_i$ and the aggregate price $\langle\lambda,R_i\rangle$ of the utilized links. If the users are distinguished only by the values $a_i$ (т.е. $\sigma_i=\sigma$), then the proportionality coefficient are common to all of them. This coefficient takes into account the total number $N$ of users.

Note that the utilities (\ref{2.8}) decrease for large values of the argument. However, the user demands $\overline x^i(\lambda)$ cannot exceed $\hat x^i$ for any price vector $\lambda$. Hence, the users consider $u_i$ only on the intervals $[0,\hat x_i]$, where these functions are increasing. 

For $f:\mathbb R\mapsto [0,\infty]$ denote by
\[ f^*(z)=\sup_{x\in\mathbb R}(xz-f(x)) \]
the conjugate function. Put
\[ -\overline u_i(y)=\begin{cases}
-u_i(y),& x\in\mathbb R_+,\\
+\infty,& \textrm{otherwise}.
\end{cases}\]
The dual utility function takes the form
\[ q(\lambda)=\langle\lambda,b\rangle+\sum_{i=1}^n (-\overline u_i)^*(-\langle R_i,\lambda)\rangle.\]

Since the function $-u_i$ are strongly convex, the elementary problems (\ref{2.4}) have the unique solutions $\overline x^i(\lambda)$ (see \cite[Theorem 5.25]{Beck2017}).
From the formula for the subdifferential of the conjugate function (see \cite[Corollary 4.21]{Beck2017}):
\[ \partial f^*(z)=\arg\max_{x}(xz-f(x))\]
it follows that the functions $(-\overline u_i)^*(z)$ are differentiable and 
\begin{equation} \label{2.9}
\frac{\partial}{\partial\lambda^j}(-\overline u_i)^*(-\langle R_i,\lambda))=-R_i^j \overline x^i(\lambda).
\end{equation}
Hence, the function $q$ is differentiable and its gradient is given by (\ref{2.5}).

The problem (\ref{1.1}),(\ref{1.2}) is solvable, since its set of feasible solutions is compact. By the strong duality theorem (see \cite[Proposition 5.3.6]{Bertsekas2009}) an optimal solution $\lambda^*$ of the dual problem (\ref{2.0}) exists and $u(x^*)=q(\lambda^*).$ Moreover, $(x^*,\lambda^*)$ is a pair of primal and dual optimal solutions, if and only if $x^*$ is feasible, $\lambda^*\ge 0$, and the optimality conditions (\ref{2.3}), (\ref{2.4}) are satisfied: \cite[Proposition 5.3.2]{Bertsekas2009}.
\begin{lemma} \label{lem:1}
An optimal solution $\lambda^*$ of the dual problem satisfies the inequalities
\[ 0\le\lambda^{*,j}\le B,\quad j=1,\dots,m.\]
Furthermore,
\[  \left|\frac{\partial q(\lambda)}{\partial\lambda^j}\right |\le\max\left\{b^j,\frac{B}{\sigma}-b^j\right\},
\quad \lambda\in [0,B]^m.\]
\end{lemma}
\begin{proof} If $\lambda^*$ is an optimal solution of the dual problem, then the unique solution $x^*$ of the primal problem satisfies the relation
\[ x^*\in\arg\max_{x\in\mathbb R^N_+} L(x,\lambda^*)\]
(see \cite[Theorem 12.11]{Beck2014a}). By (\ref{2.4}),
$x^*=\overline x(\lambda^*)$.
It follows that $x^*$ is determined by the relations 
\begin{equation} \label{2.10}
u_i'(x^{*,i})-\langle R_i,\lambda^*\rangle \le 0,\quad x^{*,i} \left(u_i'(x^{*,i})-\langle R_i,\lambda^*\rangle\right)=0,\quad 1\le i\le N
\end{equation}
(see \cite[Example 9.4]{Beck2014a}).

Let $j\in\{1,\dots,m\}$. Recall that the set $I_j=\{i:R_i^j=1\}$ of all users $i$, utilizing the link $j$, is nonempty. If $\lambda^{*,j}>B$, then 
\[ u_i'(x^{*,i})\le u_i'(0)\le B<\lambda^{*,j}\le \langle R_i,\lambda^*\rangle,\quad i\in I_j,  \]
since the functions $x_i\mapsto u_i'(x_i)$ are decreasing.
From (\ref{2.10}) it follows that $x^{*,i}=0$, $i\in I_j$. Hence,
\[\sum_{i=1}^N R_i^j x^{*,i}=\sum_{i\in I_j} R_i^j x^{*,i}=0\]
in contradiction to the complementary slackness condition (\ref{2.2}).

Furthermore, for $(N\sigma)$-strongly convex function $-u_i$ we have the inequality
\[ (-u_i'(x^i)+u_i'(y^i))(x^i-y^i)\ge N\sigma(x^i-y^i)^2,\quad x^i, y^i\in\mathbb R_+\]
(see \cite[Theorem 5.24]{Beck2017}). Put $x^i=\overline x^i(\lambda)$, $y^i=0$:
\[ (-u_i'(\overline x^i)+u_i'(0))\overline x^i\ge N\sigma (\overline x^i)^2,\]
Using the optimality condition (\ref{2.10}) for $\overline x_i(\lambda)$:
\[ \overline x^i\left(u_i'(\overline x^i)-\langle R_i,\lambda\rangle\right)=0,\]
we get
\[ \overline x^i\left(u_i'(0)-\langle R_i,\lambda\rangle\right)\ge N\sigma (\overline x^i)^2.\]
Thus,
\begin{equation} \label{2.11}
\overline x^i\le\frac{1}{N\sigma}\left(u_i'(0)-\langle R_i,\lambda\rangle\right)\le\frac{B}{N\sigma}.
\end{equation}
The inequality
\[  \left|\frac{\partial q(\lambda)}{\partial\lambda^j}\right |=\left|b^j-\sum_{i=1}^N R_i^j\overline x_i\right|\le\max\left\{b^j,\frac{B}{\sigma}-b^j\right\},\]
which follows from (\ref{2.11}), completes the proof.
\end{proof}

Lemma \ref{lem:1} shows that the minimization of the dual objective function can be performed over the hypercube $\Lambda=[0,B]^m$. Let
$\Pi_\Lambda(y)=\arg\min\{\|z-y\|^2:z\in S\}$
be an orthogonal projection onto $\Lambda$:
\[\Pi_\Lambda(y)^j=\begin{cases}
					0,& y^j\le 0,\\
					y^j, & 0\le y^j\le B,\\
					B, & y^j\ge B.
                   \end{cases}
\]
 (see \cite[Example 8.10]{Beck2014a}). 
On some probability space $(\Omega,\mathscr F,\mathsf P)$ consider a sequence $(\xi_r)_{r=1}^\infty$ of independent random variables, uniformly distributed on $\{1,\dots,N\}$:
 \[ \mathsf P(\xi_r=i)=\frac{1}{N},\quad i\in\{1,\dots,N\}.\]
Let $\mathscr F_k=\sigma(\xi_1,\dots,\xi_k)$ be a natural filtration of the process $(\xi_r)_{r=1}^\infty$. 

The recurrence formula
\begin{equation} \label{2.12}
 \lambda_{t+1}=\Pi_\Lambda\left(\lambda_{t}-\eta_t (b- N R_{\xi_{t+1}}\overline x^{\xi_{t+1}}(\lambda_{t}))\right),\quad t\ge 1,\quad \lambda_1\in S.
\end{equation}
with deterministic steps $\eta_t>0$ defines the projected stochastic gradient descent method for the problem
\[ q(\lambda)\to\min_{\lambda\in\Lambda}.\]
Indeed, since the random variable $\lambda_{t}$ is $\mathscr F_t$-measurable, the following conditional expectation can be computed by ``freezing'' $\lambda_t$:
\begin{equation} \label{2.13}
\mathsf E\left(b- N R_{\xi_{t+1}}\overline x^{\xi_{t+1}}(\lambda_t))|\mathscr F_t\right)=b-\sum_{i=1}^N  R_i\overline x^i(\lambda_t)=q'(\lambda_t).
\end{equation}

The argumentation in the proof of the following lemma is similar to
\cite[Theorem 3.1]{Hazan2016}.
\begin{lemma} \label{lem:2}
Let $\lambda_t$ be a sequence generated by the projected stochastic gradient descent method (\ref{2.12}). Then for any decreasing sequence
$\eta_t>0$ the following estimate holds true
\begin{align} 
\mathsf E q(\overline\lambda_T) \le & q(\lambda^*)+\frac{1}{T}\left(\frac{m B^2}{2}\frac{1}{\eta_T}+\frac{L^2}{2}\sum_{t=1}^T\eta_t\right),\label{2.14}\\
\overline\lambda_T:=&\frac{1}{T}\sum_{t=1}^T\lambda_t,\qquad L:=\left(\sum_{j=1}^m\max\left\{b^j,\frac{B}{\sigma}-b^j\right\}^2\right)^{1/2}. \nonumber
\end{align}
\end{lemma}
\begin{proof} Let $z_t=b-N R_{\xi_t}\overline x^{\xi_t}(\lambda_t)$, $r_t=\lambda_t-\lambda^*$. Using the inequality (\ref{2.11}), we get
\[\mathsf E\left\|z_{t+1}\right\|^2=
\mathsf E\sum_{j=1}^m(b^j-N R_{\xi_{t+1}}^j\overline x^{\xi_{t+1}})^2\le\sum_{j=1}^m\max\left\{b^j,\frac{B}{\sigma}-b^j\right\}^2=L^2.
\]
Furthermore, since by the ``Pythagorean theorem'' 
\[\|\Pi_\Lambda \mu-\lambda^*\|\le\|\mu-\lambda^*\|,\quad\mu\in\mathbb R^m\]
(see \cite[Theorem 2.1]{Hazan2016}), then 
\begin{align*}
\|r_{t+1}\|^2=&\|\lambda_{t+1}-\lambda^*\|^2=\|\Pi_\Lambda(\lambda_t-\eta_t z_{t+1})-\lambda^* \|^2\le\| \lambda_t-\eta_t z_{t+1}-\lambda^*\|^2\\
=&\|r_t\|^2-2\eta_t\langle z_{t+1},\lambda_t-\lambda^*\rangle+\eta_t^2\|z_{t+1}\|^2.
\end{align*}
Using (\ref{2.13}), we get
\begin{align*}
\mathsf E\|r_{t+1}\|^2&=\mathsf E\|r_t\|^2-2\eta_t\mathsf E\langle\mathsf E( z_{t+1}|\mathscr F_t),\lambda_t-\lambda^*\rangle+\eta_t^2\mathsf E\|z_{t+1}\|^2\\
&\le\mathsf E\|r_t\|^2-2\eta_t\mathsf E\langle q'(\lambda_t),\lambda_t-\lambda^*\rangle+\eta_t^2 L^2.
\end{align*}
By the convexity of $q$:
\[ q(\lambda^*)-q(\lambda_t)\ge \langle q'(\lambda_t),\lambda^*-\lambda_t\rangle \]
it follows that
\[ \mathsf E\|r_{t+1}\|^2\le \mathsf E\|r_t\|^2+2\eta_t \mathsf E(q(\lambda^*)-q(\lambda_t))+\eta_t^2 L^2,\]
\[\mathsf E q(\lambda_t)-q(\lambda_*)\le\frac{\mathsf E\|r_t\|^2-\mathsf E\|r_{t+1}\|^2}{2\eta_t}+\frac{L^2\eta_t}{2}.\]
After the summation and rearranging terms, we get 
\begin{align*}
\sum_{t=1}^T(\mathsf E q(\lambda_t)-q(\lambda_*))&\le\frac{1}{2}\left(\frac{1}{\eta_1}\mathsf E\|r_1\|^2 +\left(\frac{1}{\eta_2}-\frac{1}{\eta_1}\right)\mathsf E\|r_2\|^2+\dots\right.\\
&\left.+\left(\frac{1}{\eta_T}-\frac{1}{\eta_{T-1}}\right)\mathsf E\|r_T\|^2-\frac{1}{\eta_T}\mathsf E\|r_{T+1}\|^2\right)+\frac{L^2}{2}\sum_{t=1}^T\eta_t\\
&\le\frac{m B^2}{2}\frac{1}{\eta_T}+\frac{L^2}{2}\sum_{t=1}^T\eta_t.
\end{align*}
Here we used the estimates $\|r_t\|^2\le m B^2$ and the fact that $\eta_t$ is decreasing. Dividing by $T$ and using the convexity of $q$: 
\[q(\overline\lambda_T)\le\frac{1}{T}\sum_{t=1}^T q(\lambda_t),\]
we get the desired estimate (\ref{2.14}).
\end{proof}

The main result of the paper is the following theorem. Its proof uses the ideas of \cite[Theorem 1]{Beck2014}.
\begin{theorem} \label{th:1}
Define $\lambda_t$ by the recurrence formula (\ref{2.12}) with $\eta_t=K/\sqrt{t}$. Then for $\overline\lambda_T=\frac{1}{T}\sum_{t=1}^T\lambda_t$ the following estimates hold true:
\begin{align}
\mathsf E\left(\sum_{i=1}^N R_i^j \overline x^i(\overline\lambda_T)-b^j\right)^+ &\le\sqrt{\frac{2D}{\sigma}}\frac{1}{T^{1/4}},\qquad
D=\frac{m B^2}{2K}+KL^2,\label{2.15}\\
u(x^*)-\mathsf E u(\overline x(\overline\lambda_T))&\le B\sqrt{\frac{2D}{\sigma}}\frac{1}{T^{1/4}} \label{2.16}.
\end{align}
\end{theorem}
\begin{proof} Since
\[\sum_{t=1}^T\frac{1}{\sqrt t}=\sum_{t=1}^T\int_{t-1}^t\frac{du}{\sqrt t}\le \sum_{t=1}^T\int_{t-1}^t\frac{du}{\sqrt u}=\int_0^T\frac{du}{\sqrt{u}}=2\sqrt{T},\]
by substituting $\eta_t=K/\sqrt{t}$ into (\ref{2.14}), we get
\begin{equation} \label{2.17}
\mathsf E q(\overline\lambda_T)\le q(\lambda^*)+\frac{D}{\sqrt{T}}.
\end{equation}
By the Assumption \ref{as:2} the function $x\mapsto - L(x,\lambda)$ is $(N\sigma)$-strongly convex, that is, the function 
\[- L(x,\lambda)-\frac{\sigma N}{2}\|x\|^2\]
is convex. Hence,
\[\frac{N\sigma}{2}\sum_{i=1}^N (x^i-\overline x^i(\lambda))^2\le L(\overline x(\lambda),\lambda)-L(x,\lambda)\]
(see \cite[Theorem 5.25]{Beck2017}). On the other hand,
\[L(\overline x(\lambda),\lambda)-L(x,\lambda)=q(\lambda)-\sum_{i=1}^N u_i(x^i)-\left\langle\lambda, b-\sum_{i=1}^N R_i x^i\right\rangle.\]
In particular, for the optimal solution $x^*$ of the primal problem (\ref{1.1}), (\ref{1.2}) we have
\[\frac{N\sigma}{2}\sum_{i=1}^N (x^{*,i}-\overline x^i(\lambda))^2\le q(\lambda)-u(x^*)=q(\lambda)-q(\lambda^*).\]
By (\ref{2.17}) it follows that
\begin{equation} \label{2.18}
\mathsf E\sum_{i=1}^N (x^{*,i}-\overline x^i(\overline\lambda_T))^2\le\frac{2}{N\sigma}(\mathsf E q(\overline\lambda_T)-q(\lambda^*))\le \frac{2D}{N\sigma\sqrt T}.
\end{equation}

For the discrepancy in the feasibility conditions we have the estimate
\begin{align*}
 \sum_{i=1}^N R_i^j \overline x^i(\overline\lambda_T)-b^j\le& \sum_{i=1}^N R_i^j\overline x^i(\overline\lambda_T)-\sum_{i=1}^N R_i^j x^{*,i} \le \sum_{i=1}^N |\overline x^i(\overline\lambda_T)-x^{*,i}|\\
\le&\sqrt{N}\left(\sum_{i=1}^N (\overline x^i(\overline\lambda_T)-x^{*,i})^2\right)^{1/2}.
\end{align*}
Using (\ref{2.18}), we get
\[\mathsf E\left(\sum_{i=1}^N R_i^j \overline x^i(\overline\lambda_T)-b^j\right)^2\le N \sum_{i=1}^N (\overline x^i(\overline\lambda_T)-x^{*,i})^2
\le\frac{2D}{\sigma\sqrt T}.\]
This implies (\ref{2.15}). Furthermore,
\begin{align*}
\sum_{i=1}^N (u_i(x^{*,i})-u_i(\overline x^i))&\le\sum_{i=1}^N u_i'(\overline x^i)(x^{*,i}-\overline  x^i) \le B\sum_{i=1}^N |x^{*,i}-\overline  x^i|\\ 
&\le  B\sqrt{N}\left(\sum_{i=1}^N(x^{*,i}-\overline x^i)^2\right)^{1/2}.
\end{align*}
Again using (\ref{2.18}), we get the inequality
\[\mathsf E\left( \sum_{i=1}^N(u_i(x_i^*)-u_i(\overline x_i))\right)^2\le\frac{2B^2 D}{\sigma \sqrt{T}},\]
implying (\ref{2.16}).
\end{proof}

The estimates (\ref{2.15}), (\ref{2.16}) do not depend on the number $N$ of users. This qualitative result is the main point of Theorem \ref{th:1}.

Let $B/\sigma\ge 2\max_{1\le j\le m} b^j$. Then
\begin{equation} \label{2.18B}
L=\left(\sum_{j=1}^m\max\left\{b^j,\frac{B}{\sigma}-b^j\right\}^2\right)^{1/2}\le\frac{B}{\sigma}\sqrt{m}.
\end{equation}
Replace in (\ref{2.15}) the constant $L$ by its upper bound  (\ref{2.18B}):
\[ D=mB^2\left(\frac{1}{2K}+\frac{K}{\sigma^2}\right)\] 
and select the ``optimal'' constant $K$ by minimizing this expression:
\begin{equation} \label{2.18C}
K=\frac{\sigma}{\sqrt 2}.
\end{equation}
This constant will be used in computer experiments in section \ref{sec:3}.

Recall that a continuously differentiable function $f:\mathbb R^m\mapsto\mathbb R$ is called $\beta$-smooth, if
\[ \|f'(x)-f'(y)\|\le\beta\|x-y\|.\]
By Theorem 5.26 from \cite{Beck2017} the functions $\varphi_i(z):=(-\overline u)^*_i(z)$ of a single variable are $1/(N\sigma)$-smooth. By (\ref{2.9}),
\[ q'(\lambda)=b-\sum_{i=1}^N\varphi_i'(-\langle R_i,\lambda\rangle) R_i. \]
Hence the function $q$ is $m/\sigma$-smooth:
\begin{align}
\|q'(\lambda)-q'(\mu)\|&\le\sum_{i=1}^N |\varphi_i'(-\langle R_i,\lambda\rangle)- \varphi_i'(-\langle R_i,\mu\rangle )|\cdot\|R_i\|\nonumber \\
&\le \sum_{i=1}^N\frac{1}{\sigma N} |\langle R_i,\lambda-\mu\rangle| \sqrt{m}\le\frac{m}{\sigma}\|\lambda-\mu\|. \label{2.20}
\end{align}
The constant in this estimate can be refined, using the structure of the network: see \cite[Lemma III.1]{Beck2014}.

Following \cite{Beck2014}, consider the fast gradient descent method of Nesterov \cite{Nesterov1983}:
\begin{equation} \label{2.20A}
 \mu_1=\widehat\lambda_0,\quad \tau_1=1,
\end{equation} 
\begin{equation} \label{2.20B}
\widehat\lambda_t=\left[\mu_t-\frac{\sigma}{m}q'(\mu_t)\right]^+=\left[\mu_t-\frac{\sigma}{m}\left(b-\sum_{i=1}^N R_i \overline x_i(\mu_t)\right)\right]^+,
\end{equation} 
\begin{equation} \label{2.20C}
\tau_{t+1}=\frac{1+\sqrt{1+4\tau_t^2}}{2},\quad
\mu_{t+1}=\widehat\lambda_t+\frac{\tau_t-1}{\tau_{t+1}}(\widehat\lambda_t-\widehat\lambda_{t-1}),
\end{equation} 
where $t\ge 1$. Denote by $L_u=B\sqrt{N}$ the Lipschitz constant of the function $u$:
\begin{align*}
|u(x)-u(y)|&\le\sum_{i=1}^N |u_i'(z^i)(x^i-y^i)|\le \sum_{i=1}^N |u_i'(0)| |x^i-y^i|\\
&\le B \sum_{i=1}^N |x^i-y^i|\le B\sqrt N\|x-y\|
\end{align*}
and by $L_{q'}=m/\sigma$ the Lipschitz constant of the vector-function $q'$: see (\ref{2.20}). For the routing matrix  $R$ we have the estimate 
\[ \left|\sum_{k=1}^N R^j_k x^k\right|\le\sum_{k=1}^N |x^k|\le\sqrt{N}\|x\|.\]
In the notation of \cite{Beck2014} this means that 
\[\|R\|_{2,\infty}:=\max\left\{\max_j\left|\sum_{k=1}^N R^j_k x_k\right|: \|x\|\le 1\right\}\le \sqrt N.\] 
Note also that $\|\widehat\lambda_0-\lambda^*\|\le B\sqrt{m}$. Theorems 1 and 2 of \cite{Beck2014} give the following estimates:
\begin{align}
q(\widehat\lambda_T)-q(\lambda^*) &\le 2L_{q'}\frac{\|\widehat\lambda_0-\lambda^*\|^2}{(T+1)^2}\le\frac{C}{T^2},\quad C=\frac{2m^2 B^2}{\sigma},\nonumber\\
u(x^*)-u(\overline x_t(\widehat\lambda_T)) &\le L_u\sqrt{\frac{2C}{\sigma N}}\frac{1}{T}=2\frac{mB^2}{\sigma T},\label{2.27}\\
\left[\sum_{i=1}^N R_i^j \overline x_i^j(\widehat\lambda_T)-b^j\right]^+ &\le\|R\|_{2,\infty}\sqrt{\frac{2C}{\sigma N}}\frac{1}{T}=2\frac{mB}{\sigma T}.\label{2.28}
\end{align}

The fast gradient descent method will be used in section \ref{sec:3} for comparison with the projected stochastic gradient descent method (\ref{2.12}). As already mentioned in the introduction, the estimates (\ref{2.27}), (\ref{2.28}) are much better than (\ref{2.15}),(\ref{2.16})  in the order of $T$, but each iteration of the fast gradient descent method can be significantly more labor-consuming, than in the method (\ref{2.12}). 

\section{Examples} \label{sec:3}
\setcounter{section}{2}
\setcounter{equation}{0}
\begin{example} \label{ex:1}
To better understand the properties of the optimal solutions corresponding to the quadratic utility functions (\ref{2.8}), consider a network with two links and three users. Assume that the user 1 utilizes both links, and the users 2 and 3 utilize the links 1 and 2 respectively. Let the link capacities be $2$ and $1$. Thus,
\begin{equation} \label{3.0}
R_1=\begin{pmatrix}
1\\
1
\end{pmatrix},\quad 
R_2=\begin{pmatrix}
1\\
0
\end{pmatrix},\quad
R_3=\begin{pmatrix}
0\\
1
\end{pmatrix},\quad 
b=\begin{pmatrix}
2\\
1
\end{pmatrix}.
\end{equation}

This network was considered in \cite[Example 2.3]{Srikant2004} under the assumption that the user utility functions are logarithmic: $u_i(x_i)=\ln x_i$. In this case the solutions of the primal (\ref{1.1}), (\ref{1.2}) and dual (\ref{2.0}) problems look as follows
\[ x^*=\left(\frac{1}{\lambda^{*,1}+\lambda^{*,2}},\frac{1}{\lambda^{1,*}},\frac{1}{\lambda^{*,2}}\right),\quad \lambda^{*,1}=\frac{\sqrt{3}}{1+\sqrt{3}},\quad \lambda^{*,2}=\sqrt{3}.\]
Note, that
\begin{equation} \label{3.1}
 0<\lambda^{*,1}<\lambda^{*,2},\quad 0<x^{*,1}<x^{*,3}<x^{*,2}.
\end{equation} 

Now consider the users with identical utility functions (\ref{2.8}): 
\[ u_i(x^i)=a x^i-\frac{3}{2}\sigma (x^i)^2.\]
An elementary analysis of the optimality conditions (\ref{2.1}) -- (\ref{2.3}) gives the solutions, presented in Table \ref{tab:1}.
\begin{table}[h]
\def\arraystretch{1.2}
\begin{tabular}{c|c|c|c|c}
\hline
$a/\sigma$ & $(0,3/2]$ & $[3/2, 9/2]$ & $[9/2,9]$ & $[9,\infty)$\\
\hline
$\lambda^{*,1}$ &   0                 &      0                                          &    $ 2a/3-3\sigma$          & $a-6\sigma$\\ 
\hline
$\lambda^{*,2}$ &   0                 &     $a-3\sigma/2$           &   $2a/3$                       & $a-3\sigma$\\ 
\hline
$x^{*,1}$            & $a/(3\sigma)$ &     $1/2$          & $1-a/(9\sigma)$         & 0 \\ 
\hline
$x^{*,2}$           &  $a/(3\sigma)$ & $a/(3\sigma)$   &   $1+a/(9\sigma)$      & 2 \\ 
\hline
$x^{*,3}$           &  $a/(3\sigma)$ &   $1/2$              &  $a/(9\sigma)$           & 1 \\  
\hline
\end{tabular}
\vspace{\abovecaptionskip}
\caption{The dependence of the optimal solution on the parameter $a/\sigma$} \label{tab:1}
\end{table}

Note that for all values of the parameter $a/\sigma$ the inequalities (\ref{3.1}) are satisfied at least in the non-strict sense. Furthermore, fix the common value coefficient $a$ and consider $\sigma$ as a penalty, assigned by the network. If the penalty is very high: $a/\sigma\le 3/2$, then the network resources are available for free and the users select identical transmission rates,  since their utility functions are identical. If $a/\sigma\in (3/2,9/2]$, then the first link is available for free. The capacity of the second link is divided equally between the first and  third users. The second user utilizes the remaining capacity of the first link only partially. If $a/\sigma\in (9/2,9]$, then both resources are scarce. In this case the inequalities (\ref{3.1}) are literally satisfied. Finally, for very small penalty:  $a/\sigma>9$ the first user is ``eliminated from the market'', and the remaining two fully utilize the capacities of the corresponding links.
\end{example}

\begin{example}  \label{ex:2}
Consider the ``network'', containing a single link with the capacity $b>0$, utilized by large number $N$ of users with the utility functions\begin{equation} \label{3.2}
u_i(x^i)=a_i x^i-\frac{\sigma N}{2}(x^i)^2,\quad a_i\in (0,B),\quad \sigma>0.
\end{equation}
The problems (\ref{2.4}) are easily solved:
\begin{equation} \label{3.3}
 \overline x^i(\lambda)=\frac{1}{N}\frac{(a_i-\lambda)^+}{\sigma},
\end{equation}
The optimal solution $\lambda^*$ of the dual problem is the solution of the equation
\begin{equation} \label{3.4}
q'(\lambda)=b-\sum_{i=1}^N\frac{1}{N}\frac{(a_i-\lambda)^+}{\sigma}=0.
\end{equation}
The stochastic projected gradient descent method (\ref{2.12}) takes the form
\[ \lambda_{t+1}=\Pi_{[0,B]}\left[\lambda_t-\eta_t(b-N\overline x^{\xi^{t+1}})\right],
\]
\[\Pi_{[0,B]}(y)=\begin{cases}
0,& y\le 0,\\
y,& y\in (0,B),\\
A,& y\ge B.
\end{cases}\]

Let $b=5$, $\sigma=1$, $N=10^5$ and assume that the value coefficients $(a_i)_{i=1}^N$ are uniformly distributed on $(0,B)$, $B=100$. We generated $k=30$ ``populations'' of $N$ users with the utilities (\ref{3.2}). For each of $30$ problem instances the solution $\lambda^*$ of the equation (\ref{3.4}) was obtained by the bisection method {\tt optimize.bisect} from the {\tt scipy} module (Python) with the standard tolerance parameters. The average values of the optimal price, optimal aggregate utility and optimal demand for the free resource equal to
\[ \lambda^*=68.3,\quad u(x^*)=394.4,\quad \sum_{i=1}^N \overline x^i(0)=50.\]
Thus, the demand for the free resource is 10 times higher than the available capacity $b=5$. Since $a_i$ are uniformly distributed on $(0,100)$, the formula (\ref{3.3}) shows that on average almost 70\% of users get zero optimal rate $\overline x^i(\lambda^*)$.

We applied the projected stochastic gradient descent method (\ref{2.12}) with $\lambda_1=0$, $\eta_t=K/\sqrt{t}$, where $K=1/\sqrt{2}$ is defined by (\ref{2.18C}). Consider the relative errors in the optimal price, optimal aggregate demand and optimal network utility:
\begin{align}
&\varepsilon_t^1:=\frac{|\overline\lambda_t-\lambda^*|}{\lambda^*},\quad
\varepsilon_t^2:=\frac{1}{b}\left|\sum_{i=1}^N \overline x^i(\overline\lambda_t)-b\right|,\label{3.5}\\
&\varepsilon_t^3:=\frac{1}{u(x^*)}\left|\sum_{i=1}^N u_i(\overline x^i(\overline\lambda_t))-u(x^*)\right|. \label{3.6}
\end{align}
The values $\varepsilon_T^i$, averaged over the sample of 30 problems, and their maximal values for the same sample, are presented in Table \ref{tab:2}. Note that the errors in the optimal demand and utility are approximately 4 times larger than the errors in the optimal price. Note also that the number of measured transmission rates $T\le 4000$ is significantly smaller than the number of users: $T/N\le 0.04$. 

\begin{table}
\begin{tabular}{c|c|c|c|c|c|c}
\hline
The number & \multicolumn{3}{c|}{Mean relative}& \multicolumn{3}{c}{Maximal relative} \\
of iterations  & \multicolumn{3}{c|}{errors in the}& \multicolumn{3}{c}{errors in the}\\
\cline{2-7}
$T$       & prices & demand & utility & prices & demand  & utility \\
\hline
1000 & 0.0129 & 0.056 & 0.049 & 0.035 & 0.155 & 0.132 \\
\hline
2000 & 0.0078  &  0.034  & 0.029 & 0.019 & 0.082 & 0.072 \\
\hline
4000 & 0.0052  &  0.022  & 0.019 & 0.016 & 0.069 & 0.060\\
\hline
\end{tabular}
\vspace{\abovecaptionskip}
\caption{The relative errors of the projected stochastic gradient descent method} \label{tab:2}
\end{table}

The fast gradient descent method (\ref{2.20A}) -- (\ref{2.20C}) (with $\widehat\lambda_0=0$) requires only few iterations to get a comparable accuracy. For example, the mean relative error in the optimal price equals to $0.0063$ for 10 iterations. However, if the user demands are measured individually, this requires $10^6$ measurements.

For unknown $\lambda^*$ the errors (\ref{3.5}), (\ref{3.6}) are unobservable. We used very simple rule for iteration stopping, which is based on the observable quantities $\overline\lambda_t=\frac{1}{t}\sum_{k=1}^t\lambda_k$:
\begin{equation} \label{3.7}
\tau=\min\left\{t\ge 2:\frac{|\overline\lambda_t-\overline\lambda_{t-1}|}{\overline\lambda_{t-1}}<\delta\right\}.
\end{equation}
For the same sample of 30 problems the results are given in the Table \ref{tab:3}. 

\begin{table}
\def\arraystretch{1.2}
\begin{tabular}{c|c|c|c|c|c}
\hline
$\delta$ & Mean & Мaximal & Minimal & Mean & Maximal \\
& number of  & number of & number of & relative & relative\\
& iterations &  iterations  & iterations & price error  & price error\\
\hline
$10^{-7}$ & 792  & 1697 & 218 & 0.0140 & 0.045 \\
\hline
$10^{-8}$ & 2170 & 4659 & 549 & 0.0091 & 0.045  \\
\hline
$10^{-9}$ & 4954 & 11538 & 768 & 0.0058 & 0.020  \\
\hline
\end{tabular}
\vspace{\abovecaptionskip}
\caption{The relative errors for the stopping rule (\ref{3.7})} \label{tab:3}
\end{table}
\end{example}

\begin{example} \label{ex:3}
Consider a network with the routing matrix and link capacities (\ref{3.0}). In contrast to the Example \ref{ex:1}, assume that there are $N=1.2\cdot 10^5$ users. Let the users with numbers $i\in\{1,\dots,N/3\}$ utilize both links, and the users with the numbers $i\in\{N/3+1,\dots,2N/3\}$ and $i\in\{2N/3+1,\dots,N\}$ utilize the links $1$ and $2$ respectively. It is assumed that the utility functions are of the form (\ref{3.2}), where $\sigma=1$, and $a_i$ are uniformly distributed on $(0,B)$. 
 
As in Example \ref{ex:2}, we generated a sample of $k=30$ problems. For each problem 200 iterations of the fast gradient descent method (\ref{2.20A}) -- (\ref{2.20C}) were performed. The obtained vector $(\lambda_1^*,\lambda_2^*)$ is considered as an exact solution of the dual problem. The corresponding solution $x^*=\overline x(\lambda^*)$ of the primal problem gives the discrepancy in the constraints (\ref{1.2}) of order $10^{-4}$.  

Put $\langle x^*\rangle_{i={k+1}}^{k+r}=\frac{1}{r}\sum_{i=1}^r x_{i+k}^*$. For $B=100$ the computer experiments show that the users, utilizing both links, are eliminated from the market: $\langle x^*\rangle_{i=1}^{N/3}=0$. The resources are shared the by remaining users so that:
\[ \langle x^*\rangle_{i=N/3+1}^{2N/3}=5\cdot 10^{-5},\quad
   \langle x^*\rangle_{i=2N/3+1}^{N}=2.5\cdot 10^{-5}
\]
similarly to the case $a/\sigma>9$ in Example \ref{ex:1}. In this case, however, many elements $x_i^*$, $i>N/3$ are also equal to $0$. For  $B=12$ we get the following results:
\[ \langle x^*\rangle_{i=1}^{N/3}\approx 10^{-5}< \langle x^*\rangle_{i=2N/3+1}^{N}\approx 1.5\cdot 10^{-5}<\langle x^*\rangle_{i=N/3+1}^{2N/3}\approx 4\cdot 10^{-5},\]
similar to the case $a/\sigma\in (9/2,9)$ in Example \ref{ex:1}.

We applied the projected stochastic gradient descent method with $\lambda_1=0$ and $\eta_t=1/\sqrt{2 t}$, as in Example \ref{ex:2}. The errors in the prices and demand, analogous to (\ref{2.5}), are understood componentwise:
 \begin{align*}
&\varepsilon_t^{1,j}:=\frac{|\overline\lambda_t^j-\lambda^{*,j}|}{\lambda^{*,j}},\quad
\varepsilon_t^{2,j}:=\frac{1}{b^j}\left|\sum_{i=1}^N \overline x_i^j(\overline\lambda_t)-b^j\right|.
\end{align*}
The errors in the network utility are computed by the formula (\ref{3.6}). The relative errors, averaged over 30 problems of the sample, and their maximal values are given in the Tables \ref{tab:4}, \ref{tab:5}.
\begin{table}[H]
\begin{tabular}{c|c|c|c|c|c|c}
\hline
The number & \multicolumn{3}{c|}{Mean relative}& \multicolumn{3}{c}{Maximal relative} \\
of iterations  & \multicolumn{3}{c|}{errors in the}& \multicolumn{3}{c}{errors in}\\
\cline{2-7}
$T$       & prices & demand & utility & prices & demand  & utility \\
\hline
2000 & 0.104 & 0.023 & 0.011 & 0.249 & 0.060 & 0.037 \\
        & 0.022 & 0.030 &           &  0.061  & 0.083 &         \\
\hline
4000 & 0.080 & 0.018 & 0.008 & 0.247 & 0.071 & 0.032 \\
        & 0.015 & 0.021 &           &  0.037  & 0.056 &         \\
\hline
8000 & 0.048 & 0.012 & 0.006 & 0.140 & 0.031 & 0.017 \\
        & 0.012 & 0.016 &           &  0.033  & 0.036 &         \\
\hline
\end{tabular}
\vspace{\abovecaptionskip}
\caption{The relative errors of the projected stochastic gradient descent method for the network with two links, $B=12$} \label{tab:4}
\end{table}
\begin{table}[H]
\begin{tabular}{c|c|c|c|c|c|c}
\hline
The number  & \multicolumn{3}{c|}{Mean relative}& \multicolumn{3}{c}{Maximal relative} \\
of iterations  & \multicolumn{3}{c|}{errors in the}& \multicolumn{3}{c}{errors in the}\\
\cline{2-7}
$T$       & prices & demand & utility & prices & demand  & utility \\
\hline
2000 & 0.020 & 0.078 & 0.093 & 0.050 & 0.196 & 0.208 \\
        & 0.033 & 0.217 &           &  0.072  & 0.495 &         \\
\hline
4000 & 0.012 & 0.045 & 0.049 & 0.033 & 0.122 & 0.110 \\
        & 0.018 & 0.116 &           &  0.046  & 0.303 &         \\
\hline
8000 & 0.006 & 0.021 & 0.026 & 0.022 & 0.081 & 0.072 \\
        & 0.010 & 0.062 &           &  0.034  & 0.220 &         \\
\hline
\end{tabular}
\vspace{\abovecaptionskip}
\caption{The relative errors of the projected stochastic gradient descent method for the network with two links, $B=100$} \label{tab:5}
\end{table}
\end{example}
\section{Conclusion}
In this paper we used the dual projected stochastic gradient descent method for the pricing of the information transmission rates over the links of a network. The main example of the utility function is the difference between the linear utility, individual for each user, and the quadratic penalty, assigned by the network. The penalty contains a coefficient, which is proportional to the total number $N$ of users. For a class of utility functions, containing the mentioned quadratic functions, we obtained the estimates for the errors in the prices, stimulating an optimal resource allocation, and for the feasibility and  the optimal network utility errors. These estimates are uniform in $N$. We presented computer experiments, confirming that, at least for networks with small number of links, a satisfactory accuracy can be obtained by measuring a relatively small number of individual user reactions to the link prices. 

\bibliographystyle{plain}      
\bibliography{lit_eng}
\end{document}